\documentclass[10pt]{amsart}
\usepackage{amsmath, amssymb, latexsym}
\usepackage{graphics}
\usepackage{mathrsfs}
\usepackage[all, knot]{xy}
\usepackage[vcentermath,enableskew]{youngtab}
\xyoption{arc}

\newtheorem{theorem}{Theorem}[section]
\newtheorem{proposition}[theorem]{Proposition}
\newtheorem{corollary}[theorem]{Corollary}
\newtheorem{conjecture}[theorem]{Conjecture}

\theoremstyle{definition}
\newtheorem{definition}[theorem]{Definition}
\newtheorem{example}[theorem]{Example}
\newtheorem{remark}[theorem]{Remark}

\numberwithin{equation}{section}

\def\K{\mathcal K}
\def\F{\mathcal F}
\def\S{\mathcal S}

\def\g{\mathfrak g}
\def\h{\mathfrak h}
\def\Z{\mathbb Z}
\def\A{\mathbb A}
\def\C{\mathbb C}
\def\Hom{{\rm Hom}}
\def\wt{{\rm wt}}
\def\het{{\rm ht}}
\def\ch{{\rm ch}}
\def\dim{{\rm dim}}
\def\vch{\overset{\circ}{\chi}}

\newcommand{\DynkinDn}
{
\xy
{\ar@{=>} (9.5,0)*{}; (0.5,0)*{}};
(10.5,0)*{};(15,0)*{} **\dir{-};
(15,0)*{};(30,0)*{} **\dir{.};
(30,0)*{};(34.5,0)*{} **\dir{-};
{\ar@{=>} (35.5,0)*{}; (44.5,0)*{}};
(0,0)*{\circ}; (10,0)*{\circ};
(35,0)*{\circ}; (45,0)*{\circ};
(0,-3)*{0}; (10,-3)*{1};
(34,-3)*{n-1};(45,-3)*{n};
\endxy
}
\newcommand{\DynkinA}
{
\xy
{\ar@{<=>} (9.5,0)*{}; (0.5,0)*{}};
(0,0)*{\circ}; (10,0)*{\circ};
(0,-3)*{0}; (10,-3)*{1};
\endxy
}

\newcommand{\DynkinDt}
{
\xy
{\ar@{=>} (9.5,0)*{}; (0.5,0)*{}};
{\ar@{=>} (10.5,0)*{}; (19.5,0)*{}};
(0,0)*{\circ}; (10,0)*{\circ}; (20,0)*{\circ};
(0,-3)*{0}; (10,-3)*{1};  (20,-3)*{2};
\endxy
}

\newcommand{\DynkinDf}
{
\xy
{\ar@{=>} (9.5,0)*{}; (0.5,0)*{}};
{\ar@{=>} (20.5,0)*{}; (29.5,0)*{}};
(10.5,0); (19.5,0) **\dir{-};
(0,0)*{\circ}; (10,0)*{\circ}; (20,0)*{\circ}; (30,0)*{\circ};
(0,-3)*{0}; (10,-3)*{1}; (20,-3)*{2}; (30,-3)*{3};
\endxy
}
\newcommand{\ublock}[1]
{\fontsize{7}{7}\selectfont
\xy (0,2)*{}="T1"; (6,2)*{}="T2"; (4,4)*{}="T3"; (10,4)*{}="T4"; 
(0,-4)*{}="B1"; (6,-4)*{}="B2"; (4,-2)*{}="B3"; (10,-2)*{}="B4"; 
"T1"; "T2" **\dir{-};"T1"; "T3" **\dir{-};"T2"; "T4" **\dir{-};"T3"; "T4" **\dir{-};
"T1"; "B1" **\dir{-};"T2"; "B2" **\dir{-};"T4"; "B4" **\dir{-};
"B1"; "B2" **\dir{-};"B2"; "B4" **\dir{-};
(3.3,-0.7)*{#1};
\endxy
\fontsize{10}{10}\selectfont}

\newcommand{\hhblock}[1]
{
\fontsize{7}{7}\selectfont
\xy (0,1)*{}="T1"; (6,1)*{}="T2"; (4,3)*{}="T3"; (10,3)*{}="T4"; 
(0,-2)*{}="B1"; (6,-2)*{}="B2"; (4,0)*{}="B3"; (10,0)*{}="B4"; 
"T1"; "T2" **\dir{-};"T1"; "T3" **\dir{-};"T2"; "T4" **\dir{-};"T3"; "T4" **\dir{-};
"T1"; "B1" **\dir{-};"T2"; "B2" **\dir{-};"T4"; "B4" **\dir{-};
"B1"; "B2" **\dir{-};"B2"; "B4" **\dir{-};
(3,-0.5)*{#1}; 
\endxy
\fontsize{10}{10}\selectfont}

\newcommand{\patterno}
{
\fontsize{7}{7}\selectfont
\xy (-29,3)*{}="TT4"; (-33,1)*{}="TT1"; (0,1)*{}="T1"; (6,1)*{}="T2"; (4,3)*{}="T3"; (10,3)*{}="T4"; 
(-33,-2)*{}="BB1"; (0,-2)*{}="B1"; (6,-2)*{}="B2"; (4,0)*{}="B3"; (10,0)*{}="B4"; 
"TT1"; "T2" **\dir{-};"TT4"; "T4" **\dir{-};
"B4"; "T4" **\dir{-};"B2"; "T2" **\dir{-};"B1"; "T1" **\dir{-};
"T1"; "T3" **\dir{-};"T2"; "T4" **\dir{-};"T3"; "T4" **\dir{-};
"T1"+(-6,0); "T3"+(-6,0); **\dir{-};"B1"+(-6,0); "T1"+(-6,0); **\dir{-};
"T1"+(-12,0); "T3"+(-12,0); **\dir{-};"B1"+(-12,0); "T1"+(-12,0); **\dir{-};
"T1"+(-18,0); "T3"+(-18,0); **\dir{-};"B1"+(-18,0); "T1"+(-18,0); **\dir{-};
"T1"+(-24,0); "T3"+(-24,0); **\dir{-};"B1"+(-24,0); "T1"+(-24,0); **\dir{-};
"BB1"; "B2" **\dir{-};"B2"; "B4" **\dir{-};
(3,-0.5)*{0}; (-3,-0.5)*{0}; (-9,-0.5)*{0}; (-15,-0.5)*{0}; (-21,-0.5)*{0}; 
(10,69)*{}="H1";
"T4"; "H1" **\dir{-}; "T4"+(-6,0); "H1"+(-6,0) **\dir{-};
"T4"+(-12,0); "H1"+(-12,0) **\dir{-}; "T4"+(-18,0); "H1"+(-18,0) **\dir{-};
"T4"+(-24,0); "H1"+(-24,0) **\dir{-}; "T4"+(-30,0); "H1"+(-30,0) **\dir{-};
"TT4"+(0,3); "T4"+(0,3) **\dir{-};
"TT4"+(0,9); "T4"+(0,9) **\dir{-};
"TT4"+(0,19); "T4"+(0,19) **\dir{-};"TT4"+(0,25); "T4"+(0,25) **\dir{-};
"TT4"+(0,28); "T4"+(0,28) **\dir{-};"TT4"+(0,31); "T4"+(0,31) **\dir{-};
"TT4"+(0,37); "T4"+(0,37) **\dir{-};"TT4"+(0,47); "T4"+(0,47) **\dir{-};
"TT4"+(0,53); "T4"+(0,53) **\dir{-};"TT4"+(0,56); "T4"+(0,56) **\dir{-};
"TT4"+(0,59); "T4"+(0,59) **\dir{-};"TT4"+(0,65); "T4"+(0,65) **\dir{-};
(-28,-0.5)*{\cdots};
(7,4.5)*{0}; (1,4.5)*{0}; (-5,4.5)*{0}; (-11,4.5)*{0}; (-17,4.5)*{0}; (-24,4.5)*{\cdots}; 
(7,9.1)*{1}; (1,9.1)*{1}; (-5,9.1)*{1}; (-11,9.1)*{1}; (-17,9.1)*{1}; (-24,9.1)*{\cdots};
(7,18.1)*{\vdots}; (1,18.1)*{\vdots}; (-5,18.1)*{\vdots}; (-11,18.1)*{\vdots}; (-17,18.1)*{\vdots}; 
(-24,18.1)*{\cdots};
(5.8,25.1)*{n};(7.4,25.1)*{-};(8.8,25.1)*{1}; 
(-0.2,25.1)*{n};(1.4,25.1)*{-};(2.8,25.1)*{1};
(-6.2,25.1)*{n};(-4.6,25.1)*{-};(-3.2,25.1)*{1};
(-12.2,25.1)*{n};(-10.6,25.1)*{-};(-9.2,25.1)*{1};
(-18.2,25.1)*{n};(-16.6,25.1)*{-};(-15.2,25.1)*{1}; (-24,25.1)*{\cdots};
(7,29.4)*{n}; (1,29.4)*{n}; (-5,29.4)*{n}; (-11,29.4)*{n}; (-17,29.4)*{n}; (-24,29.4)*{\cdots}; 
(7,32.4)*{n}; (1,32.4)*{n}; (-5,32.4)*{n}; (-11,32.4)*{n}; (-17,32.4)*{n}; (-24,32.4)*{\cdots};
(5.8,36.7)*{n};(7.4,36.7)*{-};(8.8,36.7)*{1}; 
(-0.2,36.7)*{n};(1.4,36.7)*{-};(2.8,36.7)*{1};
(-6.2,36.7)*{n};(-4.6,36.7)*{-};(-3.2,36.7)*{1};
(-12.2,36.7)*{n};(-10.6,36.7)*{-};(-9.2,36.7)*{1};
(-18.2,36.7)*{n};(-16.6,36.7)*{-};(-15.2,36.7)*{1}; (-24,36.7)*{\cdots};
(7,45.7)*{\vdots}; (1,45.7)*{\vdots}; (-5,45.7)*{\vdots}; (-11,45.7)*{\vdots}; (-17,45.7)*{\vdots}; 
(-24,45.7)*{\cdots};
(7,52.7)*{1}; (1,52.7)*{1}; (-5,52.7)*{1}; (-11,52.7)*{1}; (-17,52.7)*{1}; (-24,52.7)*{\cdots};
(7,57.1)*{0}; (1,57.1)*{0}; (-5,57.1)*{0}; (-11,57.1)*{0}; (-17,57.1)*{0}; (-24,57.1)*{\cdots};
(7,60.1)*{0}; (1,60.1)*{0}; (-5,60.1)*{0}; (-11,60.1)*{0}; (-17,60.1)*{0}; (-24,60.1)*{\cdots};
(7,64.7)*{1}; (1,64.7)*{1}; (-5,64.7)*{1}; (-11,64.7)*{1}; (-17,64.7)*{1}; (-24,64.7)*{\cdots};
\endxy
\fontsize{10}{10}\selectfont}

\newcommand{\TTBK}[1]
{\xy (0,6)*{}="T1"; (6,6)*{}="T2"; (0,0)*{}="B1"; (6,0)*{}="B2"; 
"T1"; "T2" **\dir{-};"T1"; "B1" **\dir{-};"T2"; "B2" **\dir{-};(2.8,2.3)*{#1};\endxy}

\newcommand{\TLBK}[2]
{\fontsize{7}{7}\selectfont
\xy(0,6)*{}="T1"; (6,6)*{}="T2";(0,3)*{}="M1"; (6,3)*{}="M2";(0,0)*{}="B1"; (6,0)*{}="B2"; 
"T1"; "T2" **\dir{-};"T1"; "B1" **\dir{-};"M1"; "M2" **\dir{-};"T2"; "B2" **\dir{-};"B1"; "B2" **\dir{-};
(2.8,1.2)*{#1}; (2.8,4.2)*{#2};\endxy
\fontsize{10}{10}\selectfont}

\newcommand{\FLBK}[2]
{\fontsize{7}{7}\selectfont
\xy(0,6)*{}="T1"; (6,6)*{}="T2";(0,3)*{}="M1"; (6,3)*{}="M2";(0,0)*{}="B1"; (6,0)*{}="B2"; 
"T1"; "T2" **\dir{-};"T1"; "B1" **\dir{-};"M1"; "M2" **\dir{-};"T2"; "B2" **\dir{-};"B1"; "B2" **\dir{-};
(2.8,1.2)*{#1}; (2.8,4.2)*{#2};\endxy
\fontsize{10}{10}\selectfont}

\newcommand{\TGLBK}[2]
{\fontsize{7}{7}\selectfont
\xy(0,6)*{}="T1"; (6,6)*{}="T2";(0,3)*{}="M1"; (6,3)*{}="M2";(0,0)*{}="B1"; (6,0)*{}="B2"; 
"T1"; "T2" **\dir{-};"T1"; "B1" **\dir{-};"M1"; "M2" **\dir{-};"T2"; "B2" **\dir{-};"B1"; "B2" **\dir{-};
"M1"; "B2" **\dir{-};"M2"; "B1" **\dir{-};
(2.8,1.2)*{#1}; (2.8,4.2)*{#2};\endxy
\fontsize{10}{10}\selectfont}

\newcommand{\FGLBK}[2]
{\fontsize{7}{7}\selectfont
\xy(0,6)*{}="T1"; (6,6)*{}="T2";(0,3)*{}="M1"; (6,3)*{}="M2";(0,0)*{}="B1"; (6,0)*{}="B2"; 
"T1"; "T2" **\dir{-};"T1"; "B1" **\dir{-};"M1"; "M2" **\dir{-};"T2"; "B2" **\dir{-};"B1"; "B2" **\dir{-};
"M1"; "B2" **\dir{-};"M2"; "B1" **\dir{-};
(2.8,1.2)*{#1}; (2.8,4.2)*{#2};\endxy
\fontsize{10}{10}\selectfont}

\newcommand{\THBK}[1]
{\fontsize{7}{7}\selectfont
\xy (0,3)*{}="M1"; (6,3)*{}="M2";(0,0)*{}="B1"; (6,0)*{}="B2"; 
"M2"; "B2" **\dir{-};"M1"; "B1" **\dir{-};"M1"; "M2" **\dir{-}; "B1"; "B2" **\dir{-};
(2.8,1.2)*{#1};\endxy
\fontsize{10}{10}\selectfont}

\newcommand{\Tbone}
{\xy (0,0)*++{\TGLBK{0}{0}};
 \endxy }
 
\newcommand{\Tbtwo}
{\xy (0,0)*++{\TGLBK{0}{0}};(0,6)*++{\TTBK{1}} ;
 \endxy }  
\newcommand{\Tbthree}
{\xy (0,0)*++{\TGLBK{0}{0}};(0,6)*++{\TTBK{1}};(0,10.5)*++{\THBK{2}};
 \endxy } 
\newcommand{\Tbfour}
{\xy (0,0)*++{\TGLBK{0}{0}};(0,6)*++{\TTBK{1}};(0,12)*++{\TLBK{2}{2}};
 \endxy } 
\newcommand{\Tbfive}
{\xy (0,0)*++{\TGLBK{0}{0}};(0,6)*++{\TTBK{1}};(0,12)*++{\TLBK{2}{2}};(0,18)*++{\TTBK{1}};

 \endxy } 
\newcommand{\Tbsix}
{\xy (0,0)*++{\TGLBK{0}{0}};(0,6)*++{\TTBK{1}};(0,12)*++{\TLBK{2}{2}};(0,18)*++{\TTBK{1}};
(0,22.5)*++{\THBK{0}};
 \endxy } 
\newcommand{\Tbseven}
{\xy (0,0)*++{\TGLBK{0}{0}};(0,6)*++{\TTBK{1}};(0,12)*++{\TLBK{2}{2}};(0,18)*++{\TTBK{1}};
(0,24)*++{\TLBK{0}{0}};
 \endxy }

\newcommand{\FTBK}[1]
{\xy (0,6)*{}="T1"; (6,6)*{}="T2"; (0,0)*{}="B1"; (6,0)*{}="B2"; 
"T1"; "T2" **\dir{-};"T1"; "B1" **\dir{-};"T2"; "B2" **\dir{-};(2.8,2.3)*{#1};\endxy}

\newcommand{\FHBK}[1]
{\fontsize{7}{7}\selectfont
\xy (0,3)*{}="M1"; (6,3)*{}="M2";(0,0)*{}="B1"; (6,0)*{}="B2"; 
"M2"; "B2" **\dir{-};"M1"; "B1" **\dir{-};"M1"; "M2" **\dir{-}; "B1"; "B2" **\dir{-};
(2.8,1.2)*{#1};\endxy
\fontsize{10}{10}\selectfont}

\newcommand{\Fbone}
{\xy (0,0)*++{\FGLBK{0}{0}};
 \endxy }
 
\newcommand{\Fbtwo}
{\xy (0,0)*++{\FGLBK{0}{0}};(0,6)*++{\FTBK{1}} ;
 \endxy }  
\newcommand{\Fbthree}
{\xy (0,0)*++{\FGLBK{0}{0}};(0,6)*++{\FTBK{1}};(0,12)*++{\FTBK{2}};
 \endxy } 
\newcommand{\Fbfour}
{\xy (0,0)*++{\FGLBK{0}{0}};(0,6)*++{\FTBK{1}};(0,12)*++{\FTBK{2}};(0,16.5)*++{\FHBK{3}};
 \endxy } 
\newcommand{\Fbfive}
{\xy (0,0)*++{\FGLBK{0}{0}};(0,6)*++{\FTBK{1}};(0,12)*++{\FTBK{2}};(0,18)*++{\FLBK{3}{3}};
 \endxy } 
\newcommand{\Fbsix}
{\xy (0,0)*++{\FGLBK{0}{0}};(0,6)*++{\FTBK{1}};(0,12)*++{\FTBK{2}};(0,18)*++{\FLBK{3}{3}};
(0,24)*++{\FTBK{2}};
 \endxy } 
\newcommand{\Fbseven}
{\xy (0,0)*++{\FGLBK{0}{0}};(0,6)*++{\FTBK{1}};(0,12)*++{\FTBK{2}};(0,18)*++{\FLBK{3}{3}};
(0,24)*++{\FTBK{2}};(0,30)*++{\FTBK{1}};
 \endxy }

\newcommand{\sbk}[1]
{
\fontsize{9}{9}\selectfont
\xy (-1.5,1.5)*{}="T1"; (1.5,1.5)*{}="T2"; (-1.5,-1.5)*{}="B1"; (1.5,-1.5)*{}="B2"; 
"T1"; "T2" **\dir{-};"T1"; "B1" **\dir{-};"T2"; "B2" **\dir{-}; 
(-0.2,0)*{#1};
\endxy
\fontsize{10}{10}\selectfont
}
\newcommand{\Bhbk}[1]
{
\fontsize{5}{5}\selectfont
\xy (-1.5,1.5)*{}="T1"; (1.5,1.5)*{}="T2"; (-1.5,0)*{}="B1"; (1.5,0)*{}="B2"; 
"T1"; "T2" **\dir{-};"T1"; "B1" **\dir{-};"T2"; "B2" **\dir{-}; "B1"; "B2" **\dir{-}; 
(-0.2,0.6)*{#1};
\endxy
\fontsize{10}{10}\selectfont
}

\newcommand{\hbk}[1]
{
\fontsize{5}{5}\selectfont
\xy (-1.5,1.5)*{}="T1"; (1.5,1.5)*{}="T2"; (-1.5,0)*{}="B1"; (1.5,0)*{}="B2"; 
"T1"; "T2" **\dir{-};"T1"; "B1" **\dir{-};"T2"; "B2" **\dir{-}; 
(-0.2,0.6)*{#1};
\endxy
\fontsize{10}{10}\selectfont
}

\newcommand{\sone}
{\xy (0,0)*++{\Bhbk{0}};\endxy}
\newcommand{\stwo}
{\xy (0,0)*++{\Bhbk{0}};(0,2.8)*++{\sbk{1}};\endxy}
\newcommand{\sthree}
{\xy (0,0)*++{\Bhbk{0}};(0,2.8)*++{\sbk{1}};(0,4.5)*++{\hbk{2}};\endxy}
\newcommand{\sfour}
{\xy (0,0)*++{\Bhbk{0}};(0,2.8)*++{\sbk{1}};(0,4.5)*++{\hbk{2}};(0,6.1)*++{\hbk{2}};\endxy}
\newcommand{\sfive}
{\xy (0,0)*++{\Bhbk{0}};(0,2.8)*++{\sbk{1}};(0,4.5)*++{\hbk{2}};(0,6.1)*++{\hbk{2}};(0,8.9)*++{\sbk{1}};\endxy}
\newcommand{\ssix}
{\xy (0,0)*++{\Bhbk{0}};(0,2.8)*++{\sbk{1}};(0,4.5)*++{\hbk{2}};(0,6.1)*++{\hbk{2}};(0,8.9)*++{\sbk{1}};
     (0,10.7)*++{\hbk{0}};\endxy}
\newcommand{\sseven}
{\xy (0,0)*++{\Bhbk{0}};(0,2.8)*++{\sbk{1}};(0,4.5)*++{\hbk{2}};(0,6.1)*++{\hbk{2}};(0,8.9)*++{\sbk{1}};
     (0,10.7)*++{\hbk{0}};(0,12.3)*++{\hbk{0}};\endxy}
\newcommand{\seight}
{\xy (0,0)*++{\Bhbk{0}};(0,2.8)*++{\sbk{1}};(0,4.5)*++{\hbk{2}};(0,6.1)*++{\hbk{2}};(0,8.9)*++{\sbk{1}};
     (0,10.7)*++{\hbk{0}};(0,12.3)*++{\hbk{0}};(0,15.1)*++{\sbk{1}};\endxy}
\newcommand{\snine}
{\xy (0,0)*++{\Bhbk{0}};(0,2.8)*++{\sbk{1}};(0,4.5)*++{\hbk{2}};(0,6.1)*++{\hbk{2}};(0,8.9)*++{\sbk{1}};
     (0,10.7)*++{\hbk{0}};(0,12.3)*++{\hbk{0}};(0,15.1)*++{\sbk{1}};(0,16.9)*++{\hbk{2}};\endxy}
\newcommand{\sten}
{\xy (0,0)*++{\Bhbk{0}};(0,2.8)*++{\sbk{1}};(0,4.5)*++{\hbk{2}};(0,6.1)*++{\hbk{2}};(0,8.9)*++{\sbk{1}};
     (0,10.7)*++{\hbk{0}};(0,12.3)*++{\hbk{0}};(0,15.1)*++{\sbk{1}};(0,16.9)*++{\hbk{2}};
	 (0,18.5)*++{\hbk{2}};\endxy}

\newcommand{\Go}{
{\xy (0,-5)*++{\sone};\endxy} 
}	 
\newcommand{\Gt}{
{\xy (0,-3.4)*++{\stwo};\endxy} 
}
\newcommand{\Gto}{
{\xy (-3,-4.8)*++{\sone};(0,-3.4)*++{\stwo};\endxy} 
}
\newcommand{\Gth}{
{\xy (0,-2.6)*++{\sthree};\endxy} 
}
\newcommand{\Gtho}{
{\xy (-3,-4.8)*++{\sone};(0,-2.6)*++{\sthree};\endxy} 
}
\newcommand{\Gf}{
{\xy (0,-1.8)*++{\sfour};\endxy} 
}
\newcommand{\Gtht}{
{\xy (-3,-3.4)*++{\stwo};(0,-2.6)*++{\sthree};\endxy} 
}
\newcommand{\Gfo}{
{\xy (-3,-4.8)*++{\sone};(0,-1.8)*++{\sfour};\endxy} 
}
\newcommand{\Gfi}{
{\xy (0,-0.3)*++{\sfive};\endxy} 
}
\newcommand{\Gfio}{
{\xy (-3,-4.8)*++{\sone};(0,-0.3)*++{\sfive};\endxy} 
}
\newcommand{\Gft}{
{\xy (-3,-3.4)*++{\stwo};(0,-0.3)*++{\sfive};\endxy} 
}
\newcommand{\Gthto}{
{\xy (-6,-4.8)*++{\sone};(-3,-3.4)*++{\stwo};(0,-2.6)*++{\sthree};\endxy} 
}
\newcommand{\Gs}{
{\xy (0,0.5)*++{\ssix};\endxy} 
}
\newcommand{\Gso}{
{\xy (-3,-5)*++{\sone};(0,0.5)*++{\ssix};\endxy} 
}
\newcommand{\Gfit}{
{\xy (-3,-3.4)*++{\stwo};(0,-0.3)*++{\sfive};\endxy} 
}
\newcommand{\Gfth}{
{\xy (-3,-2.6)*++{\sthree};(0,-1.8)*++{\sfour};\endxy} 
}
\newcommand{\Gfto}{
{\xy (-6,-4.8)*++{\sone};(-3,-3.4)*++{\stwo};(0,-1.8)*++{\sfour};\endxy} 
}
\newcommand{\Gse}{
{\xy (0,1.3)*++{\sseven};\endxy} 
}
\newcommand{\Gseo}{
{\xy (-3,-4.8)*++{\sone};(0,1.3)*++{\sseven};\endxy} 
}
\newcommand{\Gst}{
{\xy (-3,-3.4)*++{\stwo};(0,0.5)*++{\ssix};\endxy} 
}
\newcommand{\Gftho}{
{\xy (-6,-5)*++{\sone};(-3,-2.6)*++{\sthree};(0,-1.8)*++{\sfour};\endxy} 
}
\newcommand{\Gfith}{
{\xy (-3,-2.6)*++{\sthree};(0,-0.3)*++{\sfive};\endxy} 
}
\newcommand{\Gfito}{
{\xy (-6,-4.8)*++{\sone};(-3,-3.4)*++{\stwo};(0,-0.3)*++{\sfive};\endxy} 
}
\newcommand{\Gei}{
{\xy (0,2.8)*++{\seight};\endxy} 
}
\newcommand{\Geio}{
{\xy (-3,-4.8)*++{\sone};(0,2.8)*++{\seight};\endxy} 
}
\newcommand{\Gset}{
{\xy (-3,-3.4)*++{\stwo};(0,1.3)*++{\sseven};\endxy} 
}
\newcommand{\Gsth}{
{\xy (-3,-2.6)*++{\sthree};(0,0.5)*++{\ssix};\endxy} 
}
\newcommand{\Gsto}{
{\xy (-6,-4.8)*++{\sone};(-3,-3.4)*++{\stwo};(0,0.5)*++{\ssix};\endxy} 
}
\newcommand{\Gfif}{
{\xy (-3,-1.8)*++{\sfour};(0,-0.3)*++{\sfive};\endxy} 
}
\newcommand{\Gfitho}{
{\xy (-6,-4.9)*++{\sone};(-3,-2.6)*++{\sthree};(0,-0.4)*++{\sfive};\endxy} 
}
\newcommand{\Gftht}{
{\xy (-6,-3.4)*++{\stwo};(-3,-2.6)*++{\sthree};(0,-1.8)*++{\sfour};\endxy} 
}
\newcommand{\Gn}{
{\xy (0,3.6)*++{\snine};\endxy} 
}
\newcommand{\Gno}{
{\xy (-3,-4.8)*++{\sone};(0,3.6)*++{\snine};\endxy} 
}
\newcommand{\Geit}{
{\xy (-3,-3.4)*++{\stwo};(0,2.8)*++{\seight};\endxy} 
}
\newcommand{\Gseth}{
{\xy (-3,-2.6)*++{\sthree};(0,1.3)*++{\sseven};\endxy} 
}
\newcommand{\Gseto}{
{\xy (-6,-4.8)*++{\sone};(-3,-3.4)*++{\stwo};(0,1.3)*++{\sseven};\endxy} 
}
\newcommand{\Gsf}{
{\xy (-3,-1.8)*++{\sfour};(0,0.5)*++{\ssix};\endxy} 
}
\newcommand{\Gstho}{
{\xy (-6,-4.9)*++{\sone};(-3,-2.6)*++{\sthree};(0,0.5)*++{\ssix};\endxy} 
}

\newcommand{\Gfitht}{
{\xy (-6,-3.4)*++{\stwo};(-3,-2.6)*++{\sthree};(0,-0.2)*++{\sfive};\endxy} 
}
\newcommand{\Gfifo}{
{\xy (-6,-4.9)*++{\sone};(-3,-1.8)*++{\sfour};(0,-0.3)*++{\sfive};\endxy} 
}

\newcommand{\Gfthto}{
{\xy (-9,-4.8)*++{\sone};(-6,-3.4)*++{\stwo};(-3,-2.6)*++{\sthree};(0,-1.8)*++{\sfour};\endxy} 
}
\newcommand{\Gte}{
{\xy(0,4.4)*++{\sten};\endxy} 
}

\newcommand{\GraphF}
{
\fontsize{16}{16}\selectfont
\scalebox{.7}{\xymatrix@R=-2pc@H=-5pc{
 & & & & & & & & &  \\
\emptyset \ar[dd]^0 & & & & & & & &\Gte \\
 & & & & & &\Gei\ar[r]^2 \ar[dr]^0 &\Gn\ar[ur]^2\ar[r]^0 &\Gno \\
\Go \ar[dd]^1 & & &\Gfi \ar[r]^0 &\Gs \ar[r]^0 &\Gse \ar[ur]^1 \ar[r]^0 &\Gseo\ar[r]^1 &\Geio\ar[ur]^2\ar[r]^1 &\Geit \\
 &\Gth \ar[r]^2 \ar[ddr]^0 &\Gf \ar[ur]^1 \ar[dr]^0 & & &\Gso\ar[r]^1 &\Gst\ar[dr]^2 \ar[r]^0 &\Gset\ar[dr]^2 \ar[r]^0 &\Gseto \\
\Gt \ar[ur]^2 \ar[dr]^0& & &\Gfo \ar[r]^1 &\Gfio \ar[r]^1 \ar[ur]^0 &\Gfit\ar[dr]^2 \ar[ur]^0 & &\Gsth\ar[dr]^2 \ar[r]^0 &\Gseth \\
 &\Gto\ar[r]^2 &\Gtho\ar[ur]^2 \ar[dr]^1 & &\Gft\ar[r]^2 \ar[dr]^0 &\Gfth\ar[r]^1 \ar[dr]^0 &\Gfith\ar[r]^2 \ar[ur]^0 &\Gfif \ar[r]^0 &\Gsf \\
 & & &\Gtht\ar[ur]^2 \ar[r]^0 &\Gthto\ar[r]^2 &\Gfto\ar[r]^2\ar[dr]^1 &\Gftho\ar[r]^1 &\Gftht\ar[dr]^1 \ar[r]^0 &\Gfthto \\
 & & & & & &\Gfito\ar[dr]^2 \ar[r]^0 &\Gsto\ar[dr]^2 &\Gfitht \\
 & & & & & & &\Gfitho\ar[dr]^2 \ar[r]^0 &\Gstho \\
 & & & & & & & &\Gfifo \\
}}
\fontsize{10}{10}\selectfont
}

\begin{document}

\title[Young walls of type $D^{(2)}_{n+1}$ and Strict partitions.]
{Young walls of type $D^{(2)}_{n+1}$ and Strict partitions.}

\author[Se-jin Oh]{Se-jin Oh$^{1,2}$} 
\address{Department of Mathematical Sciences, Seoul National University,
599 Gwanak-ro, Gwanak-gu, Seoul 151-747, Korea} \email{sj092@snu.ac.kr}

\thanks{$^1$ This work was supported by BK21 Mathematical Sciences Division.}
\thanks{$^2$ This work was supported by NRF Grant \# 2010-0010753.}


\begin{abstract}
We show that the number of reduced Young walls of type $D_{n+1}^{(2)}$ with $m$ blocks is independent of $n$ and the same as the number of strict partitions of $m$. Thus the principally specialized character $\chi_n^{\Lambda_0}(t)$ of $V(\Lambda_0)$ over $U_q(D_{n+1}^{(2)})$ can be interpreted as a generating function for strict partitions. Hence we obtain an infinite family of generalizations of Euler's partition identity.
\end{abstract}

\maketitle

\section*{Introduction}

The characters of integrable modules over quantum groups $U_q(\g)$ are important algebraic invariants which \emph{determine} the isomorphism classes of integrable modules in the sense that $M \cong N$ if and only if $\ch M=\ch N$. 
In \cite{Kash90, Kash91}, Kashiwara developed the \emph{crystal basis theory} for integrable $U_q(\g)$-modules from which a lot of combinatorial properties of integrable modules can be deduced.
For instance, using an explicit realization of crystal bases, one can compute the characters of integrable modules.

In \cite{K03}, Kang introduced the notion of \emph{Young walls} as a new combinatorial scheme for realizing the crystal bases of integrable highest weight modules over quantum affine algebras.
In that paper, it was shown that the set $\F$ of \emph{proper Young walls} has the crystal structure (induced by the Kashiwara operators $\tilde{e}_i$, $\tilde{f}_i$).
Moreover, the crystal $B(\Lambda)$ of the basic representation $V(\Lambda)$ was realized as the crystal $\K$ consisting of \emph{reduced Young walls}. 
Using these realizations, we can derive explicit formulas for the characters of level 1 highest weight modules (see \cite{HK02, KK04} for more details).

A weakly decreasing sequence of non-negative integers $\lambda=(\lambda_1,\lambda_2,\cdots)$ is called a \emph{partition} of $m$, denoted by $\lambda \vdash m$, if $|\lambda|:=\sum_i\lambda_i=m$. 
A partition $\lambda$ is called a \emph{strict partition} if all the nonzero parts are strictly decreasing and an \emph{odd partition} if all the nonzero parts are odd. Let $\mathcal{R}(m)$ (respectively, $\mathcal{Q}(m)$) be the the number of strict (respectively, odd) partitions of $m$. 
Then \emph{Euler's partition identity} states that $\mathcal{R}(m)=\mathcal{Q}(m)$ because the generating function for strict partitions and the one for odd partitions are the same (see \cite{An66}, for more details):
$$\prod_{i=1}^{\infty} (1+t^i) = \prod_{i=1}^{\infty} \dfrac{1}{1-t^{2i-1}}.$$

In this paper, we show that the \emph{principally specialized character} $\chi_n^{\Lambda_0}(t)$ of the basic representation $V(\Lambda_0)$ over $U_q(D_{n+1}^{(2)})$ can be interpreted as a generating function for strict partitions.
More precisely, the number of reduced Young walls of type $D_{n+1}^{(2)}$ with  $m$ blocks, $\mathcal{R}(m)$, coincides with the number of strict partitions of $m$.
In particular, it does \emph{not} depend on $n$, which is a rather surprising fact. 
(This fact was already discovered in \cite{NY94} using the technique of vertex operators.) 
Thus we obtain infinite families of partitions for which the Euler's partition identity hold. 
Furthermore, by defining the notion of a \emph{virtual character} for strict partitions, we show that the number of strict partitions of weight $\Lambda_0-\alpha$ is the same as the number of reduced Young walls of weight $\Lambda_0-\alpha$, which leads the the following conjecture.

\begin{conjecture} \label{Conjecture: crystal structure on Strict partition}
The set $\S$ of strict partitions has a $U_q(D_{n+1}^{(2)})$-crystal structure and it is isomorphic to the highest weight crystal $B(\Lambda_0)$ over $U_q(D_{n+1}^{(2)})$ for every $n \in \Z_{\ge 2}$.
\end{conjecture}

\vskip 1em

\noindent
{\bf Acknowledgements.} 
The author would like to express his sincere gratitude to Research Institute for Mathematical Sciences, Kyoto University for their hospitality during his visit in January, 2011. After the author was finished with this paper, he was informed by Shunsuke Tsuchioka of the result in \cite{NY94}. The author is very grateful to him for pointing it out. The author would also like to thank Prof. Seok-Jin Kang and Daehong Kim for many valuable discussions and suggestions.

\vskip 3em

\section{The quantum affine algebra $U_q(D_{n+1}^{(2)})$}

Let $I=\{ 0,1,...,n \}$ ($n \ge 2$) be the index set. 
The \emph{affine Cartan datum} $(A,P^{\vee},P,\Pi^{\vee},\Pi)$ of type $D_{n+1}^{(2)}$ consists of 

\begin{enumerate}

\item the {\it Cartan matrix}
$$A=(a_{ij})_{i,j \in I}=\left(\begin{array}{ccccccc}
2 & -1 &0 & \cdots & \cdots & 0 \\
-2 & 2 & -1 & \cdots  & \cdots &0 \\
0 & -1 & 2 & \ddots   & \cdots &0 \\
\vdots & \cdots & \ddots & \ddots & \ddots  & \vdots\\
\vdots & \cdots & \cdots &-1&2&-2 \\
0  & \cdots & \cdots &0 &-1 &2
\end{array} \right)$$

\item a free abelian group $P^{\vee}=\bigoplus_{i=0}^{n} \Z h_i \oplus \Z d $, the \emph{dual weight lattice},
\item a free abelian group $P=\bigoplus_{i=0}^{n} \Z \Lambda_i \oplus \Z \delta \subset \h^*=\C \otimes_{\Z} P^{\vee} $, the \emph{weight lattice},

\item $\Pi^{\vee}=\{h_i \mid i \in I\} \subset P^{\vee}:= \Hom(P, \Z)$, the \emph{set of simple coroots},
  
\item $\Pi = \{ \alpha_i \in P|\ i \in I  \}$, the set of \emph{simple roots},
\end{enumerate}

satisfying the following properties:

\begin{enumerate}

\item[(a)] $\langle h_i, \alpha_j \rangle = a_{ij}$ for all $i,j \in I$ and $\langle d, \alpha_j \rangle = \delta_{j0}$,

\item[(b)] $\Pi$ is linearly independent,

\item[(c)] $\langle h_j, \Lambda_i \rangle = \delta_{ij}$ for all $j \in I$, $\langle d, \Lambda_i \rangle =0$,

\item[(d)] $\langle h_j, \delta \rangle=0$ and $\langle d, \delta \rangle=1$.

\end{enumerate}

We denote by $P^{+} := \{\Lambda \in P \mid  \langle h_i,\Lambda \rangle \in \Z_{\ge 0},\  i \in I \}$ the set of \emph{dominant integral weights}. The free abelian group $Q:=\sum_{i \in I} \Z\alpha_i$ is called the \emph{root lattice} and we denote by $Q^{+}:=\bigoplus_{i \in I}\Z_{\ge 0}\alpha_i$. 
For $\alpha=\sum_{i \in I} k_i \alpha_i \in Q$, we define the \emph{height} of $\alpha$ to be $\het(\alpha):=\sum_{i \in I} k_i$.
Note that the Cartan matrix is \emph{symmetrizable};
i.e., there is a diagonal matrix $D={\rm diag}(1,2,\cdots,2,1)$ such that $DA$ is symmetric. 

Let $q$ be an indeterminate. For $i\in I$ and $m,n \in \Z_{\ge 0}$, define $$ q_0=q_n=q, \ \
q_1=\cdots=q_{n-1}=q^2, \ \
[n]_{q_i} =\frac{ {q_i}^n - {q_i}^{-n} }{ {q_i} - {q_i}^{-1} }, \ \
[n]_{q_i}! = \prod^{n}_{k=1} [k]_{q_i} , \ \
\left[\begin{matrix}m \\ n\\ \end{matrix} \right]_{q_i}=  \frac{ [m]_{q_i}! }{[m-n]_{q_i}! [n]_{q_i}! }.
$$

\begin{definition} The {\em quantum group} $U_q(D_{n+1}^{(2)})$ with a Cartan datum $(A,P^{\vee},P,\Pi^{\vee},\Pi)$ is the associative
algebra over $\C(q)$ with ${\bf 1}$ generated by $e_i,f_i$ $(i \in I)$ and
$q^{h}$ $(h \in P^{\vee})$ satisfying the following relations:

\begin{enumerate}

\item  $q^0=1, q^{h} q^{h'}=q^{h+h'} $ for $ h,h' \in P^{\vee},$

\item  $q^{h}e_i q^{-h}= q^{ \langle h,\alpha_i \rangle} e_i,
          \ q^{h}f_i q^{-h} = q^{- \langle h,\alpha_i \rangle }f_i$ for $h \in P^{\vee}, i \in I$,

\item  $e_if_j - f_je_i = \delta_{ij} \dfrac{K_i -K^{-1}_i}{q_i- q^{-1}_i }, \ \ \mbox{ where } K_i=q_i^{ h_i},$

\item  $\displaystyle \sum^{1-a_{ij}}_{k=0} \left[\begin{matrix}1-a_{ij} \\ k\\ \end{matrix} \right]_{q_i} e^{1-a_{ij}-k}_i e_j e^{k}_i = 0 \quad \text{ if } i \ne j, $

\item $\displaystyle \sum^{1-a_{ij}}_{k=0} \left[\begin{matrix}1-a_{ij} \\ k\\ \end{matrix} \right]_{q_i} f^{1-a_{ij}-k}_if_j f^{k}_i=0 \quad \text{ if }  i \ne j. $

\end{enumerate}

\end{definition}

A $U_q(\g)$-module $V$ is called a \emph{weight module} if it admits a \emph{weight space decomposition} $V=\bigoplus_{\mu \in P} V_{\mu}$, where $V_{\mu}=\{ v \in V | \ q^hv=q^{\langle h, \mu \rangle } v \text{ for all } h \in P^{\vee} \}.$ If $\dim_{\C(q)} V_{\mu} < \infty \text{ for all } \mu \in P$, we define the \emph{character} of $V$ by 
$$\chi_{D^{(2)}_{n+1}}(V)=\sum_{\mu \in P} (\dim_{\C(q)} V_{\mu} )e(\mu),$$
where $e(\mu)$ is an basis element of the group algebra $\Z[P]$ with the multiplication given by $e(\mu)e(\nu)=e(\mu+\nu) \text{ for all } \mu,\nu \in P$.

Then it is proved in \cite{Lus93}, \cite[Chapter 3]{HK02} that the category of integrable modules is semisimple with its irreducible objects being isomorphic to $V_n(\Lambda)$ for some $\Lambda \in P^+$ such that 
\begin{itemize} 
\item it is generated by a unique highest weight $v_{\Lambda}$,
\item $q^h$ acts on $v_\Lambda$ by a multiplication of $q^{\langle  h_i ,\Lambda \rangle}$ for all $h \in P^{\vee}$,
\item $e_i$ and $f_i^{\langle  h_i ,\Lambda \rangle+1}$ act trivially on $v_{\Lambda}$ for all $i \in I$,
\item it admits a weight space decomposition, 
$V_n(\Lambda)=\bigoplus_{\mu \in P}V_n(\lambda)_\mu.$ 
\end{itemize}

For $V_n(\Lambda)$ ($\Lambda \in P^+$), we set $\chi^{\Lambda}_n =\chi_{D^{(2)}_{n+1}}(V(\Lambda))$,
when there is no danger of confusion.

\begin{definition} We define the \emph{pricipally specialized character} of $V_n(\Lambda)$ as follows:
$$ \chi^{\Lambda}_{D^{(2)}_{n+1}}(t)=\chi_n^{\Lambda}(t)=\sum_{m}(\sum_{ \substack{\mu \in P \\ \het(\Lambda -\mu)=m}} \dim V_n(\Lambda)_{\mu})t^m,$$
where $t$ is the indeterminate. 
\end{definition}
Note that $\chi_n^{\Lambda}(t)$ can be derived by specializing $e(\Lambda)=1$, $e(-\alpha_i)=t$ 
($i \in I$) in $\chi_n^{\Lambda}$; i.e., 
$$\chi_n^{\Lambda}(t) = \chi_n^{\Lambda}|_{\substack{e(\Lambda)=1 \\ e(-\alpha_i)=t} }, \text{ for all } i \in I.$$

The \emph{level} of $\Lambda \in P^{+}$ is defined to be the nonnegative integer 
$\langle c, \Lambda \rangle$, where $c=h_0+2h_1+\cdots+2h_{n-1}+h_n$.
Thus the dominant integral weights of level 1 are $\Lambda_0$ and $\Lambda_n$.

Let $\A_0=\{f/g \in \C(q) \mid f,g \in \C[q],g(0) \neq 0 \}$. It is shown in \cite{Kash91} that $V_n(\Lambda)$
has a unique \emph{crystal basis} $(L(\Lambda),B(\Lambda))$, where $L(\Lambda)$ is a free $\A_0$-lattice of 
$V_n(\Lambda)$, $B(\Lambda)$ is a $\C$-basis of $L(\Lambda)/qL(\Lambda)$ and $B(\Lambda)$ has a $I$-colored
oriented graph structure induced by the \emph{Kashiwara} operators $\tilde{e}_i$, $\tilde{f}_i$ ($i \in I$).
Moreover, $B(\Lambda)$ encodes the combinatorial information of $V_n(\Lambda)$ as follows:
\begin{itemize} 
\item $B_n(\Lambda)$ admits a weight space decomposition. i.e,
$$B_n(\Lambda)= \bigsqcup_{\mu \in P} B_n(\Lambda)_\mu \text{ where } B_n(\Lambda)_\mu= B_n(\Lambda) \cap V_n(\Lambda)_\mu,$$
\item $|B_n(\Lambda)_\mu| =\dim_{\C(q)}V_n(\Lambda)_\mu=\dim_{\C}V_n(\Lambda)_\mu$.
\end{itemize}
Thus $\chi^\Lambda_n$ and $\chi_n^{\Lambda}(t)$ can be expressed as follows:
\begin{align} \label{eqn:ch in crystal}
 \chi^{\Lambda}_n= \sum_{\mu \in P} |B_n(\Lambda)_\mu|e(\mu), \ \ 
\chi_n^{\Lambda}(t)= \sum_{m}(\sum_{ \substack{\mu \in P \\ \het(\Lambda -\mu)=m}} |B_n(\Lambda)_\mu|)t^m. 
\end{align}
In particular, $B(\Lambda)$ becomes a $U_q(D^{(2)}_{n+1})$-crystal.
(See \cite{HK02, Kash93} for more details on the crystal bases theory.)

\vskip 3em

\section{Young wall realization of $B_n(\Lambda)$ of level 1}

In \cite{K03}, Kang gave a realization of level 1 highest weight crystals $B(\Lambda)$ for all classical
quantum affine algebras in terms of \emph{reduced Young walls} (see \cite{HKL04} as well). Hereafter, we briefly introduce the result in
\cite{K03} only for the type $D^{(2)}_{n+1}$. Since $V_n(\Lambda_0)$ can be identified with $V_n(\Lambda_n)$
by symmetry, we will consider the case of $\Lambda_0$ only.

Basically, the Young walls are built of colored blocks. In case of the type $D^{(2)}_{n+1}$,
there are two types of blocks whose shapes are different as follows:
\begin{itemize}
\item Unit block $\quad \ublock{*} \quad $ whose colors are $1,...,n-1$,
\item Half-height block $\quad \hhblock{*} \quad$ whose colors are $0,n$.
\end{itemize}

Then the set of Young walls is the set of blocks built on the \emph{ground-state Young wall} by the following 
rules:
\begin{enumerate}
\item[(a)] All blocks should be placed on top of the ground-state Young wall or another block.
\item[(b)] The colored blocks should be stacked in the pattern given below.
\item[(c)] Except for the right-most column, there should be no free space to the right of any blocks.
\end{enumerate}
The \emph{ground-state Young wall} $Y_{\Lambda_0}$ and the pattern in (b) are given by follows:
$${\xy (0,0)*++{\patterno};\endxy}$$
where the blocks in bottom are the ground-state Young wall $Y_{\Lambda_0}$.

A column in Young wall is called a \emph{full column} if its height is a multiple of unit length. We say a Young wall \emph{proper} if none of the full columns have same height. The part of a column consisting 
of $2$ $0$-blocks, $2$ $1$-blocks, $\ldots$, $2$ $n$-blocks is called a \emph{$\delta$-column}. 
\begin{definition} \
\begin{enumerate}
\item A column in a proper Young wall is said to contain a
{\em removable $\delta$} if we may remove a $\delta$-column
from $Y$ and still obtain a proper Young wall.
\item A proper Young wall is said to be {\em reduced}
 if none of its columns contain a removable $\delta$.
\end{enumerate}
\end{definition}

For a given Young wall $Y$, we define the {\it weight}, $\wt(Y)$, of Y as follows:
\begin{align} \label{eqn:weight on partitions}
\wt(Y)= \Lambda_0-\sum_{i \in I} a_i \alpha_i, 
\end{align}
where $a_i$ is the number of $i$-blocks  on the ground-state Young wall $Y_{\Lambda_0}$.

Let $\F_n$ be the set of all proper Young walls built on $Y_{\Lambda_0}$ and $\K_n$ be the set of all reduced proper Young walls built on $Y_{\Lambda_0}$.

\begin{theorem} \cite{K03}
\begin{enumerate}
\item $\F_n$ has a crystal structure induced by the Kashiwara operators $\tilde{e}_i$ and $\tilde{f}_i$.
\item There is a crystal isomorphism between $\K_n$ and $B_n(\Lambda)$.
\end{enumerate}
\end{theorem}

\begin{definition} \label{Def:character of set}
For a $m \in \Z{\ge 0}$, $\mu \in P$ and subset $\mathcal{A}$ of $\F_n$, we define 
\begin{enumerate}
\item $\mathcal{A}[m]$ to be the subset of $\mathcal{A}$ which has $m$ blocks on $Y_{\Lambda_0}$,
\item $\mathcal{A}[\mu]$ to be the subset of $\mathcal{A}$ consisting of Young walls with 
      weight $\mu$,
\item the {\it virtual character} $\vch_n(\mathcal{A})$ of $\mathcal{A}$ to be
$$ \vch_n(\mathcal{A})= \sum_{\mu \in P} |\mathcal{A}[\mu]|e(\mu).$$
\end{enumerate}
\end{definition}

Then the equations in $\eqref{eqn:ch in crystal}$ can be written
by  
\begin{equation}  \label{eqn:ch in YW}
\chi^{\Lambda_0}_n =\vch_n(\K_n)  =\sum_{\mu \in P} |\K_n[\mu]|e(\mu), \ \ 
\chi_n^{\Lambda_0}(t)= \sum_{m}|\K_n[m]| t^m. 
\end{equation}  

\begin{proposition} \label{Prop:Fock decomposition} \cite[Corollary 2.5]{KK04} \
$$\F_n = \bigoplus_{k \in \Z_{\ge 0}} B_n(\Lambda_0-2k(\sum_{i=0}^{n} \alpha_i))^{\oplus \mathcal{P}(k)},$$ 
where $\mathcal{P}(k)$ is the number of partitions of $k$. 
\end{proposition}

In terms of Young walls, Proposition \ref{Prop:Fock decomposition} can be interpreted as follows:
\begin{equation} \label{eqn:Fock decomposition}
|\F_n[m]| =  \displaystyle \sum_{ \substack{k \ge 0 \\ m-2(n+1)k \ge 0}} 
(|\K_{n}[m-2(n+1)k] | \times \mathcal{P}(k)).
\end{equation}

\section{principally specialized character of $V_{n}(\Lambda_0)$}

Fix $\Delta=n+1$. For a given proper Young wall $Y \in \F_n$, we can associate a partition $\lambda^Y=(\lambda^Y_1,..,\lambda^Y_m,...)  \vdash |Y|$, where $\lambda^Y_i$ is the number of blocks in $i$th-column above the ground-state wall and $|Y|=\sum_{k \in Z_{\ge 0}} \lambda^Y_k$. Thus $\F_n$ and $\K_n$ can be expressed as the sets of partitions as follows:
\begin{align*} 
&\F_n= \{ \lambda=(\lambda_1 \ge \lambda_2 \ge ... \ge \lambda_m,..) |
\lambda_i=\lambda_{i+1} \text{ implies } \lambda_i=t \Delta  \text{ for some } t \in \Z_{\ge 0}  \}. \\
&\K_n= \{ \lambda \in \F_n |
\lambda_i-\lambda_{i+1} \le 2\Delta , \text{ with equality only if } \lambda_i \neq t \Delta  \text{ for all } t \in \Z_{\ge 0}   \}. 
\end{align*} 

Let $\S_n$ be the subset of $\F_n$ consisting of strictly decreasing sequences of non-negative integers; i.e.,
$$ \S_n= \{ \lambda \in \F_n | \ \lambda_i=\lambda_{i+1} \text{ implies } \lambda_i=0  \}. $$

From Definition \ref{Def:character of set},
\begin{equation} \label{eqn:chracter of S_n}
\vch_n(\S_n[m])= \sum_{\mu \in P} | \S_n[m][\mu]|e(\mu)  \text{,   } \hspace{0.3cm} 
\vch_n(\S_n)= \sum_{\mu \in P} |\S_n[\mu]|e(\mu). 
\end{equation}

Since the set of strict partitions $\S_n$ does not depend on $n$, we will drop the subindex $n$ when we want to
emphasize the independence.
\begin{example} For $S[7]$, 
\begin{align*} 
\vch_2(S[7]) = & 3e(\Lambda_0-(3\alpha_0+2\alpha_1+2\alpha_2))+ e(\Lambda_0-(2\alpha_0+2\alpha_1+3\alpha_2)) + e(\Lambda_0-(2\alpha_0+3\alpha_1+2\alpha_2)),\\
\vch_3(S[7]) = & e(\Lambda_0-(3\alpha_0+2\alpha_1+\alpha_2+\alpha_3))+
							  e(\Lambda_0-(2\alpha_0+2\alpha_1+2\alpha_2+\alpha_3))+
							  e(\Lambda_0-(2\alpha_0+2\alpha_1+\alpha_2+2\alpha_3))+ \\
							 & e(\Lambda_0-(2\alpha_0+\alpha_1+2\alpha_2+2\alpha_3))+
							  e(\Lambda_0-(\alpha_0+2\alpha_1+2\alpha_2+2\alpha_3)).
\end{align*} 
\end{example}

We will denote by $\K^{c}_n$ (respectively, $\S^{c}_n$) the complement of $\K_n$ 
(respectively, $\S_n$) in $\F_n$. Note that the set $\S^{c}_n$ depends on $n$.
However, the following theorem tells us that the cardinality of $\S^{c}_n$ does not
depend on $n$.

\begin{theorem} \label{Theorem: Alg A,B} There are $1-1$ and onto maps given as follows:\
\begin{enumerate}
\item $\overline{\Psi_n^m}:\K^{c}_n[m] \to \displaystyle \bigsqcup_{   \substack{k > 0 \\ m-2k \Delta \ge 0}} 
(\K_{n}[m-2k\Delta]  \times \{ \lambda \vdash k \} ).$
\item $\overline{\Phi_n^m}:\S^{c}_{n}[m] \to \displaystyle \bigsqcup_{ \substack{k > 0 \\ m-2k \Delta  \ge 0}} 
(\S[m-2k\Delta] \times \{ \lambda \vdash k \} ).$
\end{enumerate}
\end{theorem}

\begin{proof}
$(1)$ Actually, the first assertion comes from $\eqref{eqn:Fock decomposition}$, directly. In this proof,
we will give an explicit $1-1$ and onto map between $\K^{c}_n[m]$ and 
$\displaystyle \bigsqcup_{ \substack{k > 0 \\ m-2k \Delta \ge 0}} 
(\K_{n}[m-2k \Delta]  \times  \{ \lambda \vdash k \} )$.
Let $\Psi^m_n$ be a map from $\K^{c}_n[m] $ to $\displaystyle \bigsqcup_{\substack{k > 0 \\ m-2k \Delta \ge 0}} \K_{n}[m-2k \Delta]$ by the following algorithm ${\mathbf A}$
\begin{itemize}
\item[(${\mathbf A}$1)] Let $\lambda \in \K^{c}_n[m]$ be given. Set $\lambda^{(0)}=\lambda$ and $l=0$.
\item[(${\mathbf A}$2)] Find maximal $i$ such that 
\begin{align} \label{condition 1}
\lambda^{(l)}_{i-1}-\lambda^{(l)}_i \ge 2 t \Delta  \text{ for some $t \in \Z_{>0}$ with equlaity hold only if } \lambda^{(l)}_{i-1}=k \Delta \text{ for some k.}
\end{align}
\item[(${\mathbf A}$3)] Among $t$'s satisfying the inequality in $\eqref{condition 1}$, choose the maximal one and say $t_l$.
Set 
$$\lambda^{(l+1)}:=(\lambda^{(l)}_1-2t_l \Delta ,\lambda^{(l)}_2-2t_l \Delta ,...,\lambda^{(l)}_{i-1}-2t_l \Delta ,\lambda^{(l)}_i,\lambda^{(l)}_{i+1},...). $$
\item[(${\mathbf A}$4)] If there is no $i$ such that 
\begin{align*} 
\lambda^{(l+1)}_{i-1}-\lambda^{(l+1)}_i \ge 2t \Delta   \text{ for some $t \in \Z_{>0}$ and
 equality holds for } \lambda^{(l+1)}_{i-1}=k \Delta \text{ for some k},
\end{align*}
define $\overline{\lambda}=\lambda^{(l+1)}$ and terminate algorithm, otherwise $l=l+1$ and go to $({\mathbf A}2)$.
\end{itemize}
Then this algorithm terminates in a finite step and one can check the following things:
\begin{itemize}
\item $k=\dfrac{|\lambda|-|\overline{\lambda}|}{2\Delta } \in \Z_{>0}$,
\item $\overline{\lambda} \in \K_n[m-2 k \Delta]$,
\item $\widehat{\lambda}_i:=\dfrac{\lambda_i-\overline{\lambda}_i}{2\Delta } \in \Z_{\ge 0}$,
\item $\widehat{\lambda}:=(\widehat{\lambda}_1,\widehat{\lambda}_2,\ldots) \vdash k$.
\end{itemize}
Thus we can get a function 
$$\overline{{\Psi}_n^m}: \ \K^{c}_n[m] \to \displaystyle \bigsqcup_{\substack{k > 0 \\ m-2k \Delta \ge 0}} 
(\K_{n}[m-2k \Delta]  \times  \{ \lambda \vdash k \} )$$
given by $\lambda \mapsto (\overline{\lambda},\widehat{\lambda})$. Then one can show that
$\overline{\Psi_n^m}$ is an $1-1$ and onto map with corresponding preimage of 
$(\overline{\lambda}',\widehat{\lambda}') \in \K_{n}[m-2k \Delta]  \times  \{ \lambda \vdash k \}$
given by
$$\lambda':=(\overline{\lambda}'_1+2 \widehat{\lambda}'_1 \Delta ,\overline{\lambda}'_2+2 \widehat{\lambda}'_2 \Delta ,...,\overline{\lambda}'_k+2 \widehat{\lambda}'_k \Delta ,..). $$

$(2)$ For a given partition $\lambda$, define the inserting $(\underbrace{k\Delta ,...,k\Delta}_{j})$ into $\lambda$ denoted by $(k\Delta)^j \hookrightarrow \lambda$ as follows:
\begin{itemize}
\item Find $i$ such that $\lambda_i \ge k\Delta  > \lambda_{i+1}$,
\item Set $(k\Delta)^j \hookrightarrow \lambda :=(\lambda_1,\lambda_2,...,\lambda_i,\underbrace{k\Delta ,...,k\Delta}_{j},\lambda_{i+1},..)$.
\end{itemize}

Let $\Phi^m_n$ be a map from $\S^{c}_n[m] $ to $\displaystyle \bigsqcup_{\substack{k > 0 \\ m-2k \Delta \ge 0}} \S[m-2k \Delta]$ by the following algorithm ${\mathbf B}$
\begin{itemize}
\item[(${\mathbf B}$1)] Let $\lambda \in \S^{c}_n[m]$ be given. Set $\lambda^{(0)}=\lambda$ and $l=0$.
\item[(${\mathbf B}$2)] Find maximal $i$ such that 
\begin{align*} 
\lambda^{(l)}_{i-1}=\lambda^{(l)}_i= \widehat{\lambda}_l \Delta \text{ for some $\widehat{\lambda}_l \in \Z_{>0}$.}
\end{align*}
\item[(${\mathbf B}$3)]Set 
$$\lambda^{(l+1)}:=(\lambda^{(l)}_1,\lambda^{(l)}_2,...,\lambda^{(l)}_{i-2},\lambda^{(l)}_{i+1},...). $$
\item[(${\mathbf B}$4)] If there is no $j$ such that 
\begin{align*} 
\lambda^{(l+1)}_{j-1}=\lambda^{(l+1)}_j= k \Delta \text{ for some $k \in \Z_{>0}$}, 
\end{align*}
define $\overline{\lambda}=\lambda^{(l+1)}$ and terminate algorithm, otherwise $l=l+1$ and go to $({\mathbf B}2)$.
\end{itemize}
Then this algorithm terminates in a finite step and one can check the following things:
\begin{itemize}
\item $k=\dfrac{|\lambda|-|\overline{\lambda}|}{2\Delta } \in \Z_{>0}$,
\item $\overline{\lambda} \in \S[m-2 k \Delta]$,
\item $\widehat{\lambda}:=(\widehat{\lambda}_l,\widehat{\lambda}_{l-1},...,\widehat{\lambda}_1) \vdash k$.
\end{itemize}
Thus we can get a function 
$$\overline{{\Phi}_n^m}: \ \S^{c}_n[m] \to \displaystyle \bigsqcup_{\substack{k > 0 \\ m-2k \Delta \ge 0}} 
(\S[m-2k \Delta]  \times  \{ \lambda \vdash k \} )$$
given by $\lambda \mapsto (\overline{\lambda},\widehat{\lambda})$. Then one can show that 
$\overline{\Phi_n^m}$ is an $1-1$ and onto map with corresponding preimage of 
$(\overline{\lambda}',\widehat{\lambda}') \in \S[m-2k \Delta]  \times  \{ \lambda \vdash k \}$
given by
$$\lambda':= (\widehat{\lambda}'_l \Delta)^2 \hookrightarrow (\widehat{\lambda}'_{l-1} \Delta)^2 \hookrightarrow \cdots (\widehat{\lambda}'_1 \Delta)^2 \hookrightarrow \overline{\lambda}'.$$
\end{proof}

\begin{corollary} \label{Cor:generalization of Euler theorem} For all $m \in \Z_{\ge 0}$ and $n \in \Z_{\ge 1}$
$$|\K_n[m]|=|\S[m]|.$$
\end{corollary}

\begin{proof}
One can show that $\S[t]=\K_n[t]$ for $0 \le t \le 2 \Delta$. The Theorem \ref{Theorem: Alg A,B} tells
us that for all $m > 2\Delta$, $|\S_n^c[m]|$ depends on the set of strict partitions of $k \in \Z_{\ge 0}$ and
$\mathcal{P}(k)$ such that $\dfrac{m-k}{2\Delta} \in \Z_{>0}$. Then by using induction on $m$, we can conclude that 
$|\S_n^c[m]|=|\K_n^c[m]|$.
Hence $$|\S[m]|=|\K_n[m]|.$$
\end{proof}
\begin{example} \ 
\begin{itemize}
\item For $n=2$, the reduced Young walls of $D_3^{(2)}$ with $8$-blocks are listed below:
$$
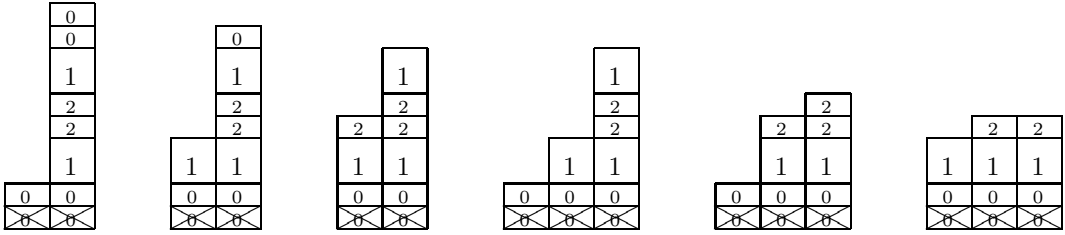
{\xy (-6,-10)*++{\Tbone};(0,2)*++{\Tbseven};\endxy}
{\xy (-6,-7)*++{\Tbtwo};(0,0.5)*++{\Tbsix};\endxy}
{\xy (-6,-5.5)*++{\Tbthree};(0,-1)*++{\Tbfive};\endxy}
{\xy (-12,-10)*++{\Tbone};(-6,-7)*++{\Tbtwo};(-0,-1)*++{\Tbfive};\endxy}
{\xy (-12,-10)*++{\Tbone};(-6,-5.5)*++{\Tbthree};(-0,-4)*++{\Tbfour};\endxy}
{\xy (-12,-7)*++{\Tbtwo};(-6,-5.5)*++{\Tbthree};(0,-5.5)*++{\Tbthree};\endxy}.
$$
\item For $n=3$, the reduced Young walls of $D_4^{(2)}$ with $8$-blocks are listed below:
$$
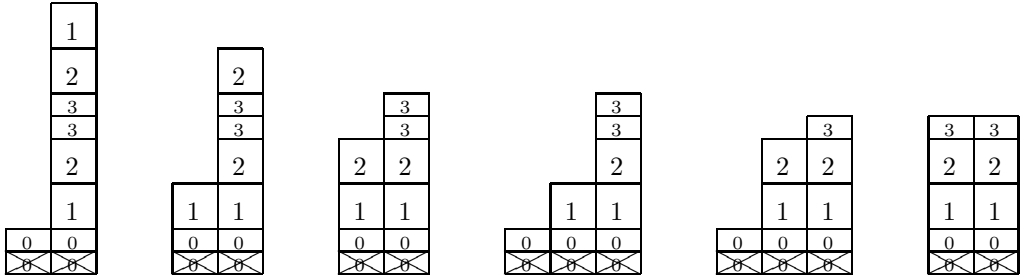
{\xy (-6,-10)*++{\Fbone};(0,5)*++{\Fbseven};\endxy}
{\xy (-6,-7)*++{\Fbtwo};(0,2)*++{\Fbsix};\endxy}
{\xy (-6,-4)*++{\Fbthree};(0,-1)*++{\Fbfive};\endxy}
{\xy (-12,-10)*++{\Fbone};(-6,-7)*++{\Fbtwo};(-0,-1)*++{\Fbfive};\endxy}
{\xy (-12,-10)*++{\Fbone};(-6,-4)*++{\Fbthree};(-0,-2.5)*++{\Fbfour};\endxy}
{\xy (-6,-2.5)*++{\Fbfour};(-0,-2.5)*++{\Fbfour};\endxy}.
$$
\end{itemize}
In this way, we can see that the number of reduced Young walls with $8$ blocks is $|\S[8]|=6$, even though $n$ is different. Note that the half-height blocks in the bottom forms the ground-state Young wall $Y_{\Lambda_0}$. 
\end{example}

Hence, for any $n \in \Z_{\ge 1}$, $\K_n$ has the same cardinality as the set of strict partitions and odd 
partitions. Moreover, combining Corollary \ref{Cor:generalization of Euler theorem} with equations $\eqref{eqn:ch in YW}$, $\eqref{eqn:Fock decomposition}$ and $\eqref{eqn:chracter of S_n}$, we can conclude that  

\begin{corollary} For all $m \in \Z_{> 0}$ and $n \in \Z_{\ge 2}$,
\begin{enumerate}
\item $\chi_n^{\Lambda_0}(t)$ is a generating function of strict partitions; i.e.,
$$ \chi_n^{\Lambda_0}(t) =\prod_{i=1}^{\infty} (1+t^i).$$
\item The number of Young walls in the set $\F_n[m]$ does not depend on $n$ and is given as follows:
$$ |\F_n[m]| = \displaystyle \sum_{\substack{k \ge 0 \\ m-2 k \Delta \ge 0}} 
(|\S[m-2k\Delta] | \times \mathcal{P}(k)).$$
\end{enumerate}
\end{corollary}


\begin{remark}
For an affine type $A^{(1)}_n$, the combinatorial realizations of crystal bases of level 1 are well-known as $n+1$-reduced colored Young diagrams. One can check that the set of $2$-reduced Young diagrams is identified with the set of strict partitions by transposing diagrams. Hence the principal specialized character of irreducible modules of level 1 over $A^{(1)}_1$ satisfies the Euler's partition identity:
\begin{equation} \label{eqn:Pch A_1^{(1)}}
\chi_{A^{(1)}_1}^{\Lambda}(t)=\prod_{i=1}^{\infty} (1+t^i)= \prod_{i=1}^{\infty} \dfrac{1}{1-t^{2i-1}}.
\end{equation}
The Dynkin diagram of $A^{(1)}_n$ and $D_{n+1}^{(2)}$ are given as follows: 
\begin{align*} 
\DynkinA \hspace{4.6em} & \text{ for } A^{(1)}_1 ,\\
\DynkinDt \hspace{3.25em} & \text{ for } D^{(2)}_3 , \\
\DynkinDf \hspace{2.03em} & \text{ for } D^{(2)}_4 , \\
\DynkinDn & \text{ for } D_{n+1}^{(2)} \ (n \ge 4).
\end{align*}
Thus our result can be interpreted as generalization of the Euler's partition identity in the sense that
the leftmost term in equation $\eqref{eqn:Pch A_1^{(1)}}$ can be replaced to the 
$\chi_{D^{(2)}_{n+1}}^{\Lambda_0}(t)$, for all $n \in \Z_{\ge 2}$.
\end{remark}

From now on, we show that the equality in Corollary \ref{Cor:generalization of Euler theorem} can be interpreted 
in more stronger sense. Then we can explain the reason why we have Conjecture \ref{Conjecture: crystal structure on Strict partition}. 

In Theorem \ref{Theorem: Alg A,B}, we defined the maps ${\Phi}^m_n$ and ${\Psi}^m_n$. By their constructions, we can observe that for $\lambda \in \F_n[m]$,
\begin{equation} \label{eqn:wt in map}
\wt({\Psi}_n^m(\lambda))=\wt({\Phi}_n^m(\lambda))=\wt(\lambda) + 2k(\sum_{n=0}^n \alpha_{i}) ,\text{ for some } k \in \Z_{\ge 0}. 
\end{equation}

\begin{theorem} \label{Thm:weights of strict partitions} For all $m \in \Z_{\ge 0}$ and $n \in \Z_{\ge 2}$,
$$ \vch_n(\S[m])=\vch_n(\K_n[m]). $$
\end{theorem}

\begin{proof}
By the definition, $\vch(\S[t])=\vch(\K_n[t])$ for $0 \le t < 2 \Delta$. For $t= 2\Delta$, one can easily check 
that $\S[t] \setminus \{ (2\Delta,0,\ldots) \} = \K_n[t] \setminus \{ (\Delta,\Delta,0,\ldots) \}$. Hence
the equation holds for $0 \le t \le 2\Delta$. From Theorem \ref{Theorem: Alg A,B} and the equation 
$\eqref{eqn:wt in map}$, for $m>2\Delta$, the $\vch(\S^c_n[m])$ depends on $k \in \Z_{\ge 0}$, the set of strict 
partitions of $k$ and $\mathcal{P}(k)$ such that $\dfrac{m-k}{2\Delta} \in \Z_{\ge 0}$. As in the similar way
of proof of Corollary \ref{Cor:generalization of Euler theorem}, 
$ \vch(\S^c_n[m])= \vch(\K^c_n[m])$ and hence our assertion holds.
\end{proof}

Thus, we conclude the following:

\begin{corollary} \label{Cor:character of S} For all $n \in \Z_{\ge 2}$,
$$\chi_n^{\Lambda_0}=\vch_n(\S).$$
\end{corollary}

From Theorem \ref{Thm:weights of strict partitions} and Corollary \ref{Cor:character of S}, one may conjecture that there are crystal structures of $B_n(\Lambda_0)$ on the set of strict partitions $\S$ for all $n \in \Z_{\ge 2}$. For example, our conjectured crystal structure on $\S$ of type $D^{(2)}_3$ is given by follows:

\begin{example} \label{Exa:GraphF} 
$$\GraphF$$
\end{example}

\bibliographystyle{amsplain}


\end{document}